\documentclass[10pt,english]{article}
\usepackage[T1]{fontenc}
\usepackage[latin9]{inputenc}
\usepackage{color}
\usepackage{amsthm}
\usepackage{amssymb}
\usepackage{graphicx}
\usepackage{setspace}
\usepackage{amsmath,amssymb}
\usepackage{bbm}
\usepackage{enumitem}
\usepackage{float}

\usepackage{pgf,tikz}
\usepackage{mathrsfs}
\usetikzlibrary{arrows}
\usetikzlibrary{decorations.pathreplacing}

\makeatletter
\numberwithin{equation}{section}
\numberwithin{figure}{section}

\newcommand\rank{\operatorname{rank}}
\newcommand\CF{{\mathcal{F}}}
\newcommand\CP{{\mathcal{P}}}
\newcommand\CS{{\mathcal{S}}}
\newcommand\CM{{\mathcal{M}}}

\newcommand\R{{\mathbb{R}}}

\newcommand\Z{{\mathbf{Z}}}

\newcommand\I{{\mathbf{I}}}
\renewcommand\P{{\mathbb{P}}}

\newcommand\Span{\operatorname{span}}

\newcommand\rowsp{\operatorname{rowsp}}
\newcommand\row{\operatorname{row}}

\makeatother

\theoremstyle{plain}
\newtheorem{theorem}{Theorem}[section]
\theoremstyle{plain}
\newtheorem{definition}[theorem]{Definition} 
 
\newtheorem{conjecture}[theorem]{Conjecture}
\newtheorem{corollary}[theorem]{Corollary}
\newtheorem{lemma}[theorem]{Lemma}
\newtheorem{proposition}[theorem]{Proposition}
\newtheorem{observation}[theorem]{Observation}
\newtheorem{claim}[theorem]{Claim}

\newtheorem{question}[theorem]{Question}

\usepackage{babel}
\begin{document}

\title{Random matrices: Probability of Normality}

\author{Andrei Deneanu \thanks{Department of Mathematics, Yale University. Email: andreiflorin.deneanu@yale.edu.} \and Van Vu \thanks{Department of Mathematics, Yale University. Email: van.vu@yale.edu ; V.Vu's research is supported by NSF grant DMS-1500944 and AFORS grant FA9550-12-1-0083.}}
\date{}
\maketitle
\begin{abstract} 
	
	We consider a random $n\times n$ matrix, $M_n$,  whose entries are independent and identically distributed (i.i.d.) Rademacher random variables (taking values  $\{ \pm1 \}$ with probability $1/2$)
	and prove 
	
	$$ 2^{-\left(0.5+o(1)\right)n^2}  \le \P (M_n \text{ is normal}) \le2^{-(0.302+o(1))n^{2}}. $$
	
\noindent 	We conjecture that the lower bound is sharp. 

\end{abstract}

\section{Introduction}
A basic notion in linear algebra is that of a {\it normal matrix}. We call an $n \times n$ real-valued matrix $A$  {\it normal} if it satisfies $AA^T=A^TA$. In this paper, we study the following question: 

\begin{question}
How often is a random matrix normal?
\end{question}

Despite the central role of normal matrices in matrix theory, to our surprise, we found no previous results concerning this natural and important question. When the entries have a  continuous distribution, the problem is, of course, easy. The probability in question is zero, as the set of normal matrices, viewed as points in $\R^{n^2}$, is not full dimensional.  However, for discrete distributions, the situation is totally different. 

We are going to focus on random matrices with all entries being i.i.d. Rademacher random variables, that is the entries take values $\pm1$ each with probability $1/2$. This is the most important  class among random matrices with discrete distribution. 
We denote the $n\times n$ Rademacher matrix by $M_n$ and by $\nu_n$ the probability that $M_n$ is normal. Throughout this paper, we assume that $n$ tends to infinity and all asymptotic notations are used under this assumption.

Clearly, the probability that $M_n$ is symmetric is $2^{-\left(0.5+o(1)\right)n^{2}}$. Since symmetric matrices are normal, 
	$$ \nu_n  \ge 2^{-\left(0.5+o(1)\right)n^{2}}.  $$
		
		We conjecture that this lower bound is sharp.
		
		\begin{conjecture} 
		
		 Let $\nu_n$ be defined as above. Then,
		 
		$$ \nu_n =  2^{-\left(0.5+o(1)\right)n^{2} }. $$
				\end{conjecture}			
				
		Our main result is that  $\nu_n \le 2^{-\left(0.302+o(1)\right)n^2}$. We actually prove a more general statement.
							
				\begin{theorem}\label{thm1}
				For any fixed $n \times n$ real valued matrix $C$  			
				$$\P(M_n M_n ^{T}=M_n ^{T}M_n +C)\leq2^{-\left(0.302+o(1)\right)n^{2}}.$$
				\end{theorem}
				
Setting $C=0$, one obtains $\nu_n\leq2^{-\left(0.302+o(1)\right)n^2}$. This more general setting plays a role in our proof. In the rest of the paper we can assume, without loss of generality, that $C$ has integer entries. 				
There have been studies of Rademacher matrices with a similar flavor, such as estimating  the probability that the matrix is singular   \cite{Komlos, KKS, TV2, BVW} or has double eigenvalues \cite{TV3, NVT}. In these cases, the conjectural bounds are of the form
$2^{-(c+o(1)) n}$, for some constant  $c >0$. While this probability is small, it is still much larger than $2^{- \Omega (n^2)} $, which enables one to exclude very rare events (those occurring with probability $2^{- \omega (n) }$) and then condition on their complement. It is, in fact, the strategy used to obtained the best current bounds 
for these problems. 

 The difficulty with the problem at hand is that we are aiming at a  bound which is extremely small (notice that any non-trivial event concerning $M_n$ holds  with probability at least $2^{-n^2}$, which is the mass of a single $\pm 1$ matrix).  There is simply no non-trivial event of probability $1 -2^{-\omega (n^2)}$ to condition on.  Thus, one needs a new strategy. The key of our approach  is  a new observation that for any given matrix, we can permute its rows and columns so that the ranks of certain submatrices follow a given pattern (see Lemma \ref{lemma2}). The fact that there are only  $n! = 2^{o(n^2)} $ permutations works in our favor and enables us  to execute a different type of conditioning. To our best knowledge, an argument of this type has not been used in random matrix theory.

\section{Preliminaries }
 
 In this section we will introduce some notation, definitions and lemmas that will be used in our proof. 

\subsection{Notation} 

\begin{definition}\label{deff}
Let $M$ be a fixed $n \times n$ matrix and let $1\leq i, j \leq n$ be fixed integers. We define $M(i;\leq j)$ to be the $i^{th}$ row of $M$ where we only keep the first $j$ entries. Similarly, we define $M(i ; \geq j)$, $M(\leq i ; j)$, $M(\geq i ; j), M(>i ; <j), M(<i; >j)$.

\begin{figure}[H]\label{newfig1}

\begin{tikzpicture}[line cap=round,line join=round,>=triangle 45,x=0.32cm,y=0.32cm]
\clip(-4,-3.6) rectangle (43,2.5);

\fill[line width=0.8pt,fill=black,fill opacity=0.05000000074505806] (4.,2) -- (4.,0) -- (43,0) -- (43,2) -- cycle;

\draw [line width=0.8pt] (4,2.)-- (43,2);
\draw [line width=0.8pt] (4,0)-- (43,0);
\draw [line width=0.8pt] (43,0)-- (43,2);
\draw [line width=0.8pt] (4,0)-- (4,2);
\draw [line width=0.8pt] (18.2,0)-- (18.2,2);

\draw [decorate,decoration={brace,amplitude=5pt,mirror,raise=2ex}]
  (4.5,0) -- (17.9,0) node[midway,yshift=-2.5em]{$M(i; \leq j)$};
  
 \draw [decorate,decoration={brace,amplitude=5pt,mirror,raise=2ex}]
  (18.5,0) -- (42.5,0) node[midway,yshift=-2.5em]{$M(i;>j)$};

\draw (-2.5,1.8) node[anchor=north west] {$M(i,>0) = $};

\draw (4.2,1.85) node[anchor=north west] {$M(i,1), $};
\draw (8.4,1.85) node[anchor=north west] {$M(i,2), $};
\draw (12.6,1) node[anchor=north west] {...};
\draw (14,1.85) node[anchor=north west] {$M(i,j)$};

\draw (18.5,1.85) node[anchor=north west] {$M(i,j+1), $};
\draw (24.5,1.85) node[anchor=north west] {$M(i,j+2), $};
\draw (30.8,1) node[anchor=north west] {...};
\draw (32.4,1.85) node[anchor=north west] {$M(i,n-1),$};
\draw (38.5,1.85) node[anchor=north west] {$M(i,n)$};

\end{tikzpicture}

\caption{A graphical representation of the notation.}
\label {fig3.2}
\end{figure}
\end{definition}

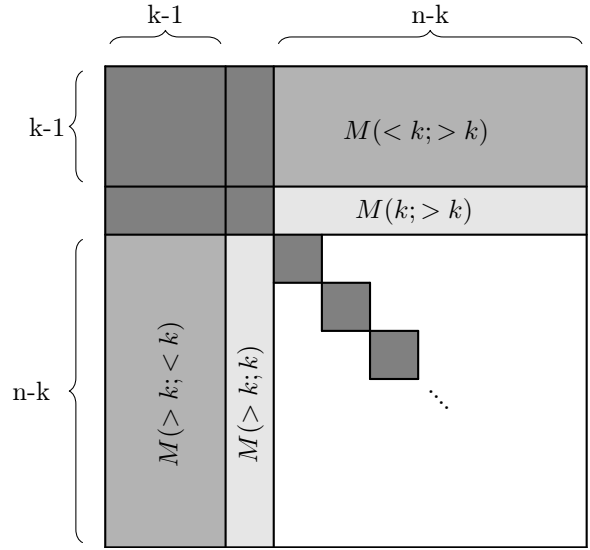
\begin{figure}[ht]\label{newfig2}

\begin{tikzpicture}[line cap=round,line join=round,>=triangle 45,x=0.32cm,y=0.32cm]
\clip(-2,0) rectangle (40,22.6);
\fill[line width=0.8pt,fill=black,fill opacity=0.00] (14.,20) -- (14.,0) -- (34,0) -- (34,20) -- cycle;
\fill[line width=0.8pt,fill=black,fill opacity=0.1] (19.,13) -- (19.,0) -- (21,0) -- (21,13) -- cycle;
\fill[line width=0.8pt,fill=black,fill opacity=0.1] (21.,15) -- (21.,13) -- (34,13) -- (34,15) -- cycle;
\fill[line width=0.8pt,fill=black,fill opacity=0.5] (14,20) -- (14,13) -- (21,13) -- (21,20) -- cycle;
\fill[line width=0.8pt,fill=black,fill opacity=0.3] (21,20) -- (21,15) -- (34,15) -- (34,20) -- cycle;
\fill[line width=0.8pt,fill=black,fill opacity=0.3] (14,13) -- (14,0) -- (19,0) -- (19,13) -- cycle;
\draw [line width=0.8pt] (14,0)-- (14,20);
\draw [line width=0.8pt] (34,0)-- (34,20);
\draw [line width=0.8pt] (14,0)-- (34,0);
\draw [line width=0.8pt] (14,20)-- (34,20);
\draw [line width=0.8pt] (21,0)-- (21,20);
\draw [line width=0.8pt] (19,0)-- (19,20);
\draw [line width=0.8pt] (14,15)-- (34,15);
\draw [line width=0.8pt] (14,13)-- (34,13);
\draw [line width=0.8pt] (21,11)-- (25,11);
\draw [line width=0.8pt] (23,9)-- (23,13);
\fill[line width=0.8pt,fill=black,fill opacity=0.5] (21,13) -- (23,13) -- (23,11) -- (21,11) -- cycle;
\draw [line width=0.8pt] (25,11)-- (25,7);
\draw [line width=0.8pt] (23,9)-- (27,9);
\fill[line width=0.8pt,fill=black,fill opacity=0.5] (23,11) -- (25,11) -- (25,9) -- (23,9) -- cycle;
\draw [line width=0.8pt] (27,9)-- (27,7);
\draw [line width=0.8pt] (25,7)-- (27,7);
\fill[line width=0.8pt,fill=black,fill opacity=0.5] (25,9) -- (27,9) -- (27,7) -- (25,7) -- cycle;
\draw [decorate,decoration={brace,amplitude=5pt, raise=2ex}]
  (14.2,20) -- (18.8,20) node[midway,yshift=2em]{k-1};
\draw [decorate,decoration={brace,amplitude=5pt, raise=2ex}]
  (21.2,20) -- (33.8,20) node[midway,yshift=2em]{n-k};
\draw [decorate,decoration={brace,amplitude=5pt, raise=2ex}]
  (14,15.2) -- (14,19.8) node[midway,xshift=-2.2em]{k-1};
\draw [decorate,decoration={brace,amplitude=5pt, raise=2ex}]
  (14,0.2) -- (14,12.8) node[midway, xshift=-2.8em]{n-k};
 \draw (27,7) node[anchor=north west] {\rotatebox{-45}{....}};
\draw (24,14.9) node[anchor=north west] {$M(k; >k)$};
\draw (23.5,18.2) node[anchor=north west] {$M(<k; >k)$};
\draw (19.1,8.8) node[anchor=north west] {\rotatebox{90}{$M(>k; k)$}};
\draw (15.8 ,9.8) node[anchor=north west] {\rotatebox{90} {$M(>k; <k)$}};
\end{tikzpicture}

\caption{A graphical representation of $M_n$.}
\end{figure}

Let us reveal our motivation behind these definitions. Notice that if we condition on the entries in the main diagonal and the first $k-1$ rows and columns of $M_n$,  then in order for $M_n$ to be normal, its entries must satisfy the following linear equation:

%\begin{equation}\label{eq3}
% 	C^T_{k-1}(k:) \cdot c_k(k:)-R_{k-1}(k:) \cdot r_k^T(k:)=c,
%\end{equation}

\begin{equation}\label{eq3}
 	M_n(>k; < k)^T \cdot M_n(>k;k)-M_n(<k;>k) \cdot M_n(k;>k)^T=c,
\end{equation}

\noindent  where $c$ is a vector in $\mathbb Z^{k-1}$, determined by the entries that were conditioned upon. We can rewrite \eqref{eq3} in a nicer way, as 

%\begin{equation} \label{eq44}
%	\left[\begin{array}{c} R_{k-1}^T(k:)\\C_{k-1}(k:) \end{array}\right]^T \left[\begin{array}{c} -r_k^T(k:)\\ \;\;\;c_k(k:) \end{array}\right]=c.
%\end{equation}

\begin{equation} \label{eq44}
	\left[\begin{array}{c} M_n(<k;>k)^T \\M_n(>k; < k)\;\; \end{array}\right]^T \left[\begin{array}{c} -M_n(k;>k)^T\\ \;\;\;M_n(>k;k) \;\;\,\end{array}\right]=c.
\end{equation}

As we will mainly be working with equations of the form  \eqref{eq44}, we define 
	$$ T_{k-1} := \left[\begin{array}{c} M_n(<k;>k)^T \\M_n(>k; < k) \;\; \end{array}\right] \text{ and }$$
	$$ x_k:= \left[\begin{array}{c} -M_n(k;>k)^T\\ \;\;\;M_n(>k;k) \;\;\; \end{array}\right]. $$

Relation \eqref{eq44} can then be rewritten as
	
\begin{equation}\label{eq4}
	T_{k-1}^Tx_k=c.
	\end{equation}
	
Given a deterministic matrix $M$, the matrices $T_i$ are well defined. Note that the number of solutions to equation \eqref{eq4} depends on the rank of $T_{k-1}$. We define $\rank_i(M)$ by

$$\rank_i (M) := \rank (T_i) . $$ 

In order to deal with the number of solutions to equations similar to equation \eqref{eq4} we need the following simple observation of Odlyzko (see \cite{Odl}). The proof follows from the simple fact that for any vector in a $k$ dimensional space, there is a set of $k$ coordinates that determines all other coordinates. 
\begin{lemma} \label{claim1}
Let $Q_{n}$ be the set of vertices of the hypercube
$\{\pm1\}^{n}$. Then for any $k$-dimensional subspace $\CS$ of $\mathbb{R}^{n}$, we have:
	$$|Q_{n}\cap{\CS}|\leq2^{k}.$$
\end{lemma}

\begin{corollary} \label{lemma1}
Let $M\in\mathbb{M}_{k\times m}(\pm1)$ be a fixed matrix of rank $r>0$, let $c\in\mathbb{M}_{k\times 1}(\mathbb{Z})$ be a fixed vector and let $x_m\in \{\pm1\}^m$ be a random vector uniformly distributed over the sample space. Then the following holds
	$$\P(Mx_m=c)\leq2^{-r}.$$
\end{corollary}

An essential element in our proof is that there exists a way to permute the rows and columns of any matrix that changes the rank of specific submatrices, in particular the ranks of $T_i$'s, while preserving the normality status of the matrix. 

\begin{definition}
Let $S_{n}$ be the set of all permutations of $\{1,2,...,n\}$. For any $\sigma\in S_{n}$ and any $n\times n$ matrix $M$, set 
	$$M_{\sigma}:= S_{\sigma}MS_{\sigma}^T,$$
where $S_{\sigma}$ is the permutation matrix associated with $\sigma$. In other words, $M_{\sigma}$ is created by permuting the rows and columns of $M$ according to $\sigma$.
\end{definition}

\subsection{Permutation Lemma} 

We  form the following equivalence classes. For two square matrices $M$ and $N$ of size $n$

$$ M\longleftrightarrow N \iff \exists\;\sigma\in S_{n}\mbox{ such that }M_{\sigma}=N.$$

\begin{definition}
Let $C$ be a fixed $n \times n$ matrix. We say that $M$ is \textit{C-normal} if and only if there exists $\sigma \in S_n$ such that $MM^T-M^TM=C_{\sigma}$.
\end{definition}

\begin{proposition}\label{prop1}
Let $\sigma\in S_{n}$, then \textbf{$M$} is $C$-normal if and only if $M_{\sigma}$ is $C$-normal. 
\end{proposition}

\begin{proof}  
For any permutation $\sigma' \in S_n$ let $S_{\sigma'}$ be the permutation matrix associated with it. Then
	\begin{align*}
	\text{$M$ is $C$-normal} & \iff \exists \, \rho \in S_n \text{ such that }MM^T-M^TM=C_{\rho}\\
	&\iff S_{\sigma}MM^TS_{\sigma}^T - S_{\sigma}M^TMS_{\sigma}^T=S_{\sigma}C_{\rho}S_{\sigma}^T\\
	&\iff S_{\sigma}MS_{\sigma}^TS_{\sigma}M^TS^T_{\sigma} - S_{\sigma}M^TS_{\sigma}^TS_{\sigma}MS_{\sigma}^T=S_{\sigma}S_{\rho}CS^T_{\rho}S_{\sigma}^T \\
	&\iff M_{\sigma}(M_{\sigma})^T - (M_{\sigma})^TM_{\sigma}=C_{\sigma\rho}\\
	&\iff \text{ $M_{\sigma}$ is $C$-normal.}
	\end{align*}
\end{proof}

\begin{observation}
Proposition \ref{prop1} implies that if $M\longleftrightarrow N $ and $M$ is \textit{$C$-normal}, then $N$ is also \textit{$C$-normal}. There are $n!=2^{\theta\left(n\log(n)\right)}=2^{o(n^{2})}$
permutations in $S_{n}$, hence it is enough to bound the equivalence classes containing \textit{$C$-normal} matrices. It follows that Theorem \ref{thm1} can be rephrased as Theorem \ref{thm10} below.
\end{observation}

\begin{theorem}\label{thm10}For any fixed matrix $C$
	$$ \P(\text{There exists } \sigma \in S_n \text{ such that }M_{n, \sigma}\text{ is C-normal}) \leq 2^{-(0.302+o(1))n^2}.$$
\end{theorem}

From now on we will say that the matrix $M$ and $N$ are equivalent if they are in the same equivalence class. 
The key idea of our argument  is that given any  matrix $M$, we can find a permutation $\sigma$ such that we can tightly control $\rank_i(M_{\sigma})$. In particular we want $\rank_i(M_{\sigma})$ to be as big as possible, so that we have many restrictions on $x_{i+1}$ in equation \eqref{eq4}.

We claim that for any matrix $M_n$ there exist constants $k,t$ and $\sigma \in S_n$ such that for all $1\leq i \leq n$, $\rank_i(M_{n,\sigma})$ equals $R_{k,t}(i)$, where $R_{k,t}(i)$  is defined below (see also Figure \ref{fig3.1}).

\begin{figure}[H]
    \includegraphics[width=14.5cm]{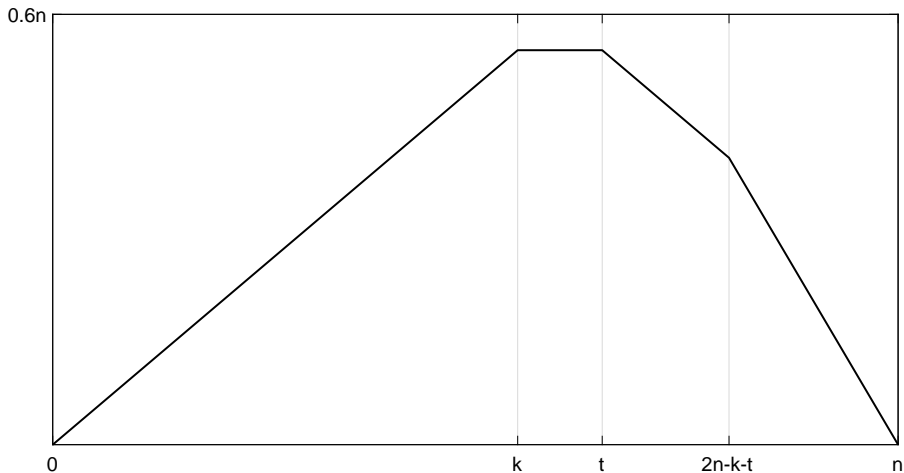}
  \caption{A graphical representation of $R_{k,t}(i)$.}
  \label{fig3.1}
\end{figure}

\begin{equation}\label{RRR}
R_{k,t} (i)= 
\begin{cases}
i & \mbox{if }0<i\leq k\\
k & \mbox{if }k<i\leq t\\
k+t-i & \mbox{if }t<i\leq 2n-k-t\\
2n-2i & \mbox{if }2n-k-t<i\leq n.
\end{cases}
\end{equation}

\noindent We are now ready to state our permutation lemma.

\begin{lemma}[Permutation Lemma]\label{lemma2} 
Let $M$ be a fixed $n \times n$ matrix. Then there exist $k,t\in \mathbb Z_+$ and $\sigma \in S_n$ such that $M_{\sigma}$ satisfies the condition \eqref{RANK} below:
\begin{equation} \label{RANK}
\rank_i (M_{\sigma}) = R_{k,t}(i), \,\,\forall\,1\le i\le n.
\end{equation}
\end{lemma}

\begin{proof}[Proof of the Permutation Lemma] $ $

Let $$ x_k':= \left[\begin{array}{c} M(k;>k)^T\\ M(>k;k)\;\;\, \end{array}\right]. $$

Note that $x_k$ and $x_k'$ differ only by the sign of the first $n-k$ entries. Recall that $T_i$ is a $2\left(n-i-1\right) \times i $ matrix. By the definition of $T_i$,  we obtain  $T_{i}$ from $T_{i-1}$  in two steps (see Figure \ref{fig3.2}):

\begin{figure}[ht]\label{fig3.2}

\begin{tikzpicture}[line cap=round,line join=round,>=triangle 45,x=0.32cm,y=0.32cm]
\clip(-1.8683269820069491,-1.55149788766866) rectangle (76.1674357222626,22.94445973526012);

\fill[line width=0.8pt,fill=black,fill opacity=0.2] (4.,20.) -- (4.,19.) -- (13.,19.) -- (13.,20.) -- cycle;
\fill[line width=0.8pt,fill=black,fill opacity=0.2] (4.,10.) -- (4.,9.) -- (13.,9.) -- (13.,10.) -- cycle;

\fill[line width=0.8pt,fill=black,fill opacity=0.2] (41.5,19.) -- (41.5,1.) -- (43.,1.) -- (43.,19.) -- cycle;

\fill[line width=0.8pt,fill=black,fill opacity=0.05] (4.,20.) -- (4.,0.) -- (13.,0.) -- (13.,20.) -- cycle;
\fill[line width=0.8pt,fill=black,fill opacity=0.05] (33.,19.) -- (33.,1.) -- (43.,1.) -- (43.,19.) -- cycle;
\fill[line width=0.8pt,fill=black,fill opacity=0.05] (18.5,19.02) -- (18.5,1.02) -- (27.,1.02) -- (27.,19.) -- cycle;
\draw [line width=0.8pt] (4.,20.)-- (4.,0.);
\draw [line width=0.8pt] (4.,0.)-- (13.,0.);
\draw [line width=0.8pt] (13.,0.)-- (13.,20.);
\draw [line width=0.8pt] (13.,20.)-- (4.,20.);
\draw [line width=0.8pt] (4.,10.)-- (13.,10.);
\draw [line width=0.8pt] (18.5,19.)-- (18.5,1.);
\draw [line width=0.8pt] (18.5,1.)-- (27.,1.);
\draw [line width=0.8pt] (27.,1.)-- (27.,19.);
\draw [line width=0.8pt] (27.,19.)-- (18.5,19.);
\draw [line width=0.8pt] (18.5,10.)-- (27.,10.);
\draw [line width=0.8pt] (33,19.02)-- (33,1.02);
\draw [line width=1.2pt, dash pattern=on 5pt off 5pt] (41.5,1.)-- (41.5,19.);
\draw [line width=0.8pt] (33,10.)-- (43,10.);
\draw [line width=0.8pt] (33,19.)-- (43,19.);
\draw [line width=1.2pt,dash pattern=on 5pt off 5pt] (4,9.)-- (13,9.);
\draw [line width=1.2pt,dash pattern=on 5pt off 5pt] (4,19.)-- (13,19.);
\draw [line width=0.8pt] (43,19.)-- (43,1.0);
\draw [->,line width=0.8pt] (14.,10.) -- (17.5,10.);
\draw [->,line width=0.8pt] (28.5,10.) -- (32.,10.);
\draw (7.2,22.7) node[anchor=north west] {$T_{i-1}$};
\draw (37.3,21.7) node[anchor=north west] {$T_{i}$};
\draw (4.8,15.7) node[anchor=north west] {$M(<i; >i)^T$};
\draw (5,5.7) node[anchor=north west] {$M(<i;>i)$};
\draw (18.2,15.7) node[anchor=north west] {$M(<i;>i+1)^T$};
\draw (18.4, 6.7) node[anchor=north west] {$M(<i;>i+1)$};
\draw [line width=0.8pt] (43,10.)-- (40.5,10.);
\draw (32.8,15.8) node[anchor=north west] {$M(\leq i; >i+1)^T$};
\draw (33.1,6.7) node[anchor=north west] {$M(\leq i; >i+1)$};
\draw [line width=0.8pt] (33.,1.)-- (43,1.);
\end{tikzpicture}

\caption{A graphical representation of the process of creating $T_i$ from $T_{i-1}$.}
\label {fig3.2}
\end{figure}
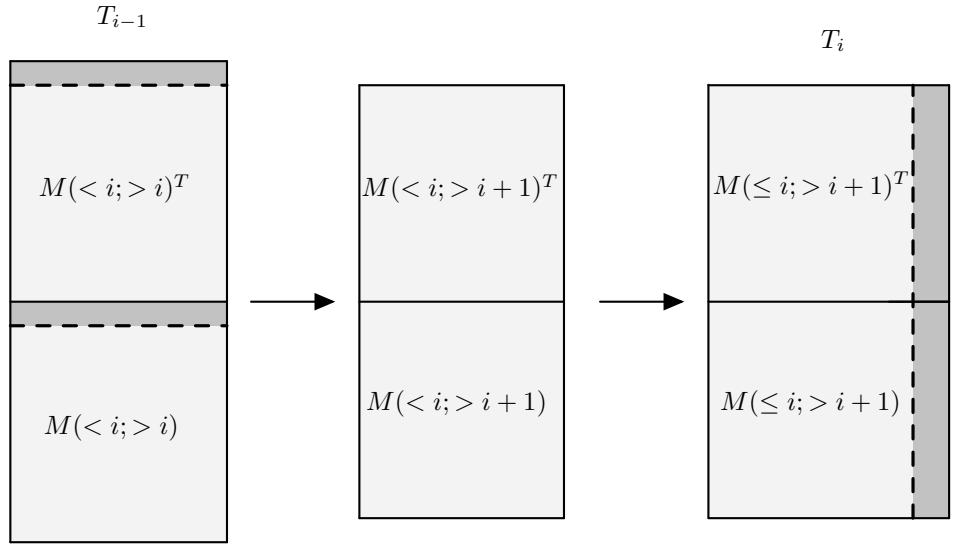

\begin{itemize}
\item First,  we delete the first and the $(n-i)^{th}$ row of $T_{i-1}$. This decreases the rank by at most $2$. Call this reduction $b_i$.

\item Next, we augment  $T_{i-1}$ by $x_i'$. This increases  the rank by at most $1$. Call this increment $a_i$.
\end{itemize}

\noindent Thus, we have 
\begin{equation}
\rank(T_i)=\rank(T_{i-1}) - b_i + a_i,
\end{equation}
where $a_i\in\{0,1\}$ and $b_i\in\{0,1,2\}$.

The desired permutation $\sigma$ is defined as the product of $n$ transpositions $\sigma: =\prod_{i=1}^n \sigma_i$, where $\sigma_i = (i, s_i) $ for some index $s_i \ge i $.
Let $M^{[i]} := M_{\prod_{k=1}^i  \sigma_k }$. We have $M_{\sigma} = M^{[n]}$, and $M^{[i]} $ is obtained from $M^{[i-1]} $ by applying $\sigma_i$, namely, swapping the $i^{\text{th}}$ row and column with the 
$s_i^{\text{th}}$ row and column.

 The index $s_i $ is defined to basically maximize the quantity $a_i - b_i$ (the gain in the rank). We use the following algorithm:
 
 \begin{enumerate}
  	\item Find indices $j \ge i$ such that the transposition $(i,j)$ minimizes the $b_i$ of $M^{[i-1]}_{(i,j)}$;  
	\item Among those $j$, pick one which maximizes $a_i $ (if we still have ties, we pick the smallest index);
	\item Set $s_i = j$ which implies $\sigma_i := (i,s_i)$ as desired.
\end{enumerate}
		
Our claim is an easy consequence of Lemma \ref{ab} below.

\end{proof}

\begin{lemma} \label{ab}  The sequence $b_i$ (of $M_\sigma$) is non-decreasing. In other words, there are indices $1 \le j_1 \le j_2 \le n$ such that 
	$b_i =0 $ for  $i \in [1,j_1-1] $, $b_i=1$ for $i \in [j_1, j_2 -1]$, and $b_i =2$ for $i \in [j_2, n]$.  Furthermore, the sequence 
	$a_i$ is non-increasing  for $i \in [1, j_1-1]$, $i \in [j_1, j_2-1] $ and $i \in [j_2-1, n]$. 
		\end{lemma} 	
		
\begin{proof}[Proof of Lemma \ref{ab}]
The intuition behind this proof is that as $i$ increases, the dimensions of the $T_i$'s work in our favor.  Let  $\sigma' : = \sigma \circ (i,i+1)$, where $(i,i+1)$ is the transposition which swaps $i$ and $i+1$ and let 
		$$ r_j(>s) : = M(j;>s) \text{ and } c_j(>s) := M(>s;j) \text{ for any }j>0.$$ 

The lemma follows from two keys observations:  

\begin{itemize}
	
	\item $b_i \le b_{i+1}$ for all $1\le i <n$.\newline
	
	Suppose $b_i > b_{i+1}$. We will show that in this case, the matrix $M_{\sigma'}$ would have a strictly smaller value associated to $b_i$, leaving $a_1,...,a_{i-1}$ and $b_1,...,b_{i-1}$ the same.
	$ $ \newline
	$ $ \newline
Consider the first row of $T_{i+1}$, that is $c_{i+3}^T(\leq i+1)$. We show that if $c_{i+3}^T(\leq i+1)$ belongs to the span of the rows of $T_{i+1}$ when the first and the $n-i-1^{th}$ rows are excluded, then $c_{i+3}^T(\leq i)$ belongs to the span of the rows of $T_{i}$ when the second and the $n-i+1^{th}$ rows are excluded. To see this, consider the following implications:
	 \begin{align*}
	 	&c_{i+3}^T(\leq i+1)   \in \Span \left(  \{ r_l(\leq i+1) \big | l \geq i+3 ,   l \not = i+2\} \cup \{ c_l^T(\leq i+1) \big | l \geq i+3 \}  \right) \Longrightarrow \\
		& \;\;\;\;\;\;\;\;\;\;\;\;\;\;\;(\text{we drop the last coordinate from each vector})\\
		& c_{i+3}^T(\leq i)  \in \Span \left(  \{ r_l(\leq i) \big | l \geq i+3, l \not = i+2 \} \cup \{ c_l^T(\leq i) \big | l \geq i+3 \}  \right) \Longrightarrow \\
		& \;\;\;\;\;\;\;\;\;\;\;\;\;\;\; (\text{we add two more vectors to the space})\\
		& c_{i+3}^T(\leq i)  \in \Span \left(  \{ r_l(\leq i+1) \big | l \geq i+2, l \not = i+2 \} \cup \{ c_l^T(\leq i) \big | l \geq i+1, l \not= i+2\}  \right).	 
	\end{align*}	
	
		This exact argument works if we replace $c^T_{i+3}(\leq i+1)$ with $r_{i+3}^T(\leq i+1)$. That is, if $b_{i+1} < b_i$, then $M_{\sigma'}$ has a strictly smaller value associated to $b_i$, which is a contradiction. 
	
	\item If $b_i=b_{i+1}$, then $a_i \ge a_{i+1}$.\newline
	
	Suppose $b_i = b_{i+1}$ and $a_i < a_{i+1}$, that is $a_i = 0$ and $a_{i+1} = 1$. Similarly, we will show that the matrix $M_{\sigma'}$ would have a strictly bigger value associated to $a_i$, leaving $a_1,...,a_{i-1}$ and $b_1,b_2,...,b_{i}$ the same. Firstly, note that by the first part argument, $M_{\sigma'}$ will generate a smaller (or equal) value for $b_i$. \newline 
	 Consider the last column of $T_{i+1}$, that is $  \left[\begin{array}{c} r_{i+1}^T(>i+2)\\c_{i+1}(>i+2) \end{array} \right] $. Since $a_{i+1} = 1$ it means that
	
		$$  \left[\begin{array}{c} r_{i+1}^T(>i+2)\\c_{i+1}(>i+2) \end{array} \right]  \not \in \Span \left \{ \left[\begin{array}{c} r_{1}^T(>i+2)\\c_{1}(>i+2) \end{array} \right], \left[\begin{array}{c} r_{2}^T(>i+2)\\c_{2}(>i+2) \end{array} \right], ..., \left[\begin{array}{c} r_{i}^T(>i+2)\\c_{i}(>i+2) \end{array} \right] \right \} $$
		
which, by adding two coordinates to every vector and deleting the last vector of the span implies

		$$  \left[\begin{array}{c} r_{i+1}^T(>i+1)\\c_{i+1}(>i+1) \end{array} \right]  \not \in \Span \left \{ \left[\begin{array}{c} r_{1}^T(>i+1)\\c_{1}(>i+1) \end{array} \right], \left[\begin{array}{c} r_{2}^T(>i+1)\\c_{2}(>i+1) \end{array} \right],..., \left[\begin{array}{c} r_{i-1}^T(>i+1)\\c_{i-1}(>i+1) \end{array} \right] \right \}.$$
	
In other words, $M_{\sigma'}$ has a strictly bigger value associated to $a_i$, which is a contradiction. 
\end{itemize}

\end{proof}

\begin{observation}\label{obs1}
Let us make some observations about this algorithm. 
	\begin{itemize}
		\item Firstly, as $T_i$ has dimensions $2(n-i-1) \times i$, note that if $\rank(T_i) < n-i-2$, we can always find a permutation $\sigma = (i,k)$ with $k\geq i$ such that 
		$$ \rank\left(T_i(M_n)\right) = \rank \left(T_{i}(M_{n,\sigma})\text{ with its first and the $(n-i)^{th}$ rows removed} \right).$$
		This implies that $b_1 = b_2 = ... =b_{n-k-2} = 0$.
		\item Secondly, as $\rank(T_t) < \rank(T_{t+1})$, it means that $b_{t+1} \geq 1$. Together with the first observation implies that:
			$$ t \geq n-k-2.$$
		\item Thirdly, since $a_i - b_i \in \{ -2, -1 , 0 ,1\}$ we have: 
			$$ k \leq 2n/3, \text{ as } \rank(T_k) = k \text{ and } \rank (T_n) = 0,\;\;\;\,$$
			$$ t+ k/2 \leq n, \text{ as } \rank(T_t) = k \text{ and } \rank (T_n) = 0,$$
			 $$ t+ k \geq n, \text{ as } \rank(T_t) = k \text{ and } \rank (T_n) = 0.\;\;\;$$
	\end{itemize}
\end{observation}
$ $

\begin{definition}
We define $\CM_{k,t}(C)$ to be the set of all C-normal matrices $M$ with $\pm 1$ entries which satisfy condition \eqref{RANK} from Lemma \ref{lemma2}
\end{definition}

\subsection {A recursion} 
In this subsection, we use Lemma \ref{lemma2} to derive a recursive bound. Let $C$ be a fixed $n \times n$ matrix. In the rest of the paper $C$ will remain fixed and, for simplicity, we will write $\CM_{k,t}$ instead of $\CM_{k,t}(C)$. The following lemma allows us to exploit the fact that if $M$ is in the form of equation \eqref{RANK}, we can control $\P(M \text{ is normal})$. To do this, we use a conditional argument. We condition on the elements of the main diagonal and on the elements on the first $i$ rows and columns of $M$ for various $i$. We denote the set of these elements by $D_i$, that is, 
	$$D_i : = \{ M(i',i') \text{ where } 1\leq i' \leq n\} \cup \{ M(i',j') \text{ where either } i' \leq i  \text{ or } j' \leq i  \} .$$
	Recall that:
	$$ x_i:= \left[\begin{array}{c} -M_n(i;>i)^T\\ \;\;\;M_n(>i;i) \;\;\; \end{array}\right]. $$
which means that $D_i$ is uniquely determined by the entries of $x_i$ and the elements of $D_{i-1}$.

\begin{lemma}[Recursion Lemma]\label{lemma10}
For any $1\leq k \leq t \leq n$ and $1\leq i \leq n$, we have

$$\sup_{D_{i-1}} \P(M_n \in \CM _{k,t}|D_{i-1})\leq 2^{-\min \left(R_{k,t}(i-1),2n-2i\right) } \cdot \sup\limits_{D_i}  \P(M_n \in \CM_{k,t}|D_i).$$
\end{lemma}

\begin{proof}

Note that if $M\in \mathcal M_{k,t}$ then $M$ is $C$-normal, and so by relation \eqref{eq4} 
	\begin{equation}\label{andrei}
	T^T_{i-1}x_i=c,
	\end{equation}
	
\noindent where $c$ is a  vector  uniquely determined by $D_{i-1}$. Thus, conditioned on $D_{i-1}$, $x_i$ belongs to a subspace $H$ of dimension $\max \{ 2n-2i-\rank(T_{i-1}), 0 \} $.  
Recall that by the permutation lemma, $ \rank (T_{i-1}) =  \rank_{i-1}(M_{\sigma}) = R_{k,t}(i-1)$.
Using Lemma \ref{claim1} and the independence of the entries of $M_n$, we have 

\begin{align}\label{eq22}
	 \P(M_n \in \CM_{k,t}|D_{i-1}) &=  \sum_{h\in H \cap \{\pm\}^{2n-2i} } \P(M_n \in \CM _{k,t}|D_i) \cdot \P(x_i = h) \nonumber \\ 
	 & = \sum_{h\in H \cap \{\pm\}^{2n-2i}  }  \P(M_n \in \CM _{k,t}|D_i) \cdot \P(x_i = h) \cdot \I(h \text{ satisfies } \eqref{andrei} )\nonumber \\
	 & \leq \sup_{D_i} \P(M_n \in \CM_{k,t}|D_i) \cdot \sum_{h \in H \cap \{\pm\}^{2n-2i} } 2^{-(2n-2i)} \cdot \I(h \text{ satisfies }\eqref{andrei}) \nonumber \\
	& \hspace{0cm}\leq 2^{-(2n-2i)+\max(2n-2i-R_{k,t}(i-1) ,0)} \sup_{D_i} \P(M_n \in \CM_{k,t}|D_i) \nonumber \\
	& \hspace{0cm} = 2^{-\min \left(R_{k,t}(i-1),2n-2i\right) } \cdot \sup\limits_{D_i}  \P(M_n \in \CM_{k,t}|D_i)
	\end{align}
	
\end{proof}

\section{Proof of Theorem 1.3}

Note that 
	$$ \P(M_n \text{ is C-normal }) \leq \sum_{k,t} \P(M_n \in \CM_{k,t}),$$ 
as $\cup_{k,t} \CM_{k,t}$ contains all C-normal matrices. Our goal is to bound $\P(M_n \in{\CM}_{k,t})$ for each $k,t \in \Z_+$. Note that for some specific values of $k$ and $t$, the problem is trivial. One can easily see from observation \ref{obs1} that $\CM_{k,t}$ is empty when $k+t<n$, $k > 2n/3$ or $t+k/2>n$.

The proof goes as follows: in sections \ref{s1} and \ref{s2} we present two different approaches, each providing bounds for different $k$ and $t$. In Section \ref{s3} we combine the two results to get the desired bound through an optimization process.

\subsection{The First Case}\label{s1}

\begin{lemma} \label{lemma3}
We have, for $1\leq k \leq \frac{2n}{3}$ and $\frac{k}{2}<n-t \leq k$

\begin{equation}\label{important}
\P(M_n \in\CM_{k,t})\leq
\begin{cases}
		2^{n^2+k^2+t^2+kt-2kn-2nt+o(n^2)}  & \mbox{if }k\geq \frac {n}{2}\\
 	        2^{t^2-3k^2+2kn+kt-2nt+o(n^2)}  & \mbox{if }k\leq \frac {n}{2}.
	 
\end{cases}
\end{equation}

In particular we have:
\begin{equation}\label{alpha1}
\P(M_n \text{ is C-normal}\,)\leq 2^{-0.25n^2+o(n^2)}.
\end{equation}
\end{lemma}

\begin{proof}[Proof of Lemma \ref{lemma3}]
$ $

We keep the notation from the Lemma \ref{lemma10}, that is 
	$$D_i : = \{ M(i',i') \text{ where } 1\leq i' \leq n\} \cup \{ M(i',j') \text{ where either } i' \leq i  \text{ or } j' \leq i  \} .$$

$ $
\begin{align*}
\P(M_n \in \CM_{k,t})   &=\sum_{D_0} \P(M_n \in \CM_{k,t}|D_0)\P(D_0)\\
					&\leq \sup_{D_0}\P(M_n \in \CM_{k,t}|D_0).
\end{align*}
By applying repeatedly Lemma \ref{lemma10},  and using Observation \ref{obs1} we have
\begin{align*}
	\sup_{D_0} \P(M_n \in \CM_{k,t}| D_0) & \leq 2^{-\min \left(R_{k,t}(0),2n-2\right) } \cdot \sup\limits_{D_1}  \P(M_n \in \CM_{k,t}|D_1)\\ 
						& \leq 2^{- R_{k,t}(1) - R_{k,t}(0)} \cdot \sup\limits_{D_2}  \P(M_n \in \CM_{k,t}|D_2)\\
						& ... \\
						& \leq 2^{-\sum_{i=0}^{2n-k-t-1} R_{k,t}(i)} \cdot \sup\limits_{D_{2n-k-t}}\P(M_n \in \CM_{k,t}|D_{2n-k-t})\\
						& \leq 2^{-\sum_{i=0}^{2n-k-t-1} R_{k,t}(i)}\cdot 2^{2(k+t-n)} \cdot \sup\limits_{D_{2n-k-t+1}}\P(M_n \in \CM_{k,t}|D_{2n-k-t+1})\\
						& ...\\
						& \leq 2^{-\sum_{i=0}^{2n-k-t-1} R_{k,t}(i)}\, \cdot 2^{-2\sum_{i=0}^{k+t-n} (k+t-n-i)}\\
						&\leq 2^{-\sum_{i=0}^{2n-k-t-1} R_{k,t}(i)}\,2^{-(k+t-n)^2+o(n^2)}.
\end{align*}

\noindent We conclude that
\begin{align}\label{bound}
	\P(M_n \in \CM_{k,t}) & \leq 2^{-(k+t-n)^2 -\sum_{i=1}^{2n-k-t}R_{k,t}(i)+o(n^2)}\nonumber\\
					 & \leq 2^{-(k+t-n)^2-k^2/2-(t-k)k-(3k/2+t-n)(2n-k-2t)+o(n^2)}\nonumber\\
					 & \leq 2^{n^2+k^2+t^2+kt-2kn-2nt+o(n^2)}.
\end{align}

\noindent If $k\leq \frac{n}{2}$, then the bound from \eqref{bound} is weak so we use a slightly different approach. Suppose that $M\in \CM_{k,t}$. By Observation \ref{obs1} we have $t\geq n-k$ so Lemma \ref{ab} implies:
	$$a_{k+1}-b_{k+1}=a_{k+2}-b_{k+1}=...=a_{n-k}-b_{n-k}=0.$$ 
	By Observation \ref{obs1} we also know that if $p\leq n-k-2$, then $b_p=0$ which implies that
	$$a_{k+1}=a_{k+2}=...=a_{n-k-2}=0.$$ 
	
\noindent It follows that we do not change the column spaces of $T_p$ for $k\leq p\leq n-k-2$ and, in particular,
	$$ x_i':= \left[\begin{array}{c} M_n(p;>p)^T\\ M_n(>p;p) \;\; \end{array}\right]\in \text{Column space} (T_{k}) \text{ for any $k < p \leq n-k-2$}.$$
\noindent Let $G$ denote the column space of $T_k$, that is $G$ is a $k$-dimensional space. Using Lemma \ref{claim1} for any $k<p\leq n-k-2$ and the independence of the entries of $M_n$, we have:
	\begin{align*}
	\sup_{D_{p-1}} \P( M_n \in \CM_{k,t}|D_{p-1}) & \leq \sum_{h\in G} \P(M_n \in \CM_{k,t}|D_p)\P( x_i' = h)\\
								     & \leq 2^{k-2(n-p-1)} \sup_{X_p\in \{\pm 1 \}^{2(n-p)}} \P(M_n\in \CM_{k,t}|D_p).
	\end{align*}

Now we can combine this result with Lemma \ref{lemma10}:
\begin{align}
	\P(M_n\in \CM_{k,t}|D)    &\leq 2^{-\sum_{i=0}^{k-1} R_{k,t}(i) }\sup_{D_k} \P(M_n\in \CM_{k,t}|D_k)\nonumber\\
						&\hspace{-1cm}\leq 2^{-k^2/2+o(n^2)} 2^{\sum_{i=k}^{n-k-2}(k-2(n-i-1))} \sup_{D_{n-k-1}} \P(M_n \in \CM_{k,t}|D_{n-k-1})\nonumber\\
						&\hspace{-1cm}\leq 2^{-n^2-5k^2/2+3nk+o(n^2)} 2^{-\sum_{i=n-k-1}^{2n-k-t-1}R_{k,t}(i)}\sup_{D_{2n-k-t}} \P(M_n\in \CM_{k,t}|D_{2n-k-t})\nonumber\\
						&\hspace{-1cm}\leq 2^{n^2-2k^2+2t^2-4nt+3kt+o(n^2)}2^{-(k+t-n)^2+o(n^2)}\nonumber\\
						&\hspace{-1cm}\leq 2^{t^2-3k^2+2kn+kt-2nt+o(n^2)}.
\end{align}

Note that if we maximize the bounds over all possible choices of $k$ and $t$ we conclude:
	$$ \P(M_n\in \CM_{k,t})\leq 2^{-0.25n^2+o(n^2)} $$
and the conclusion follows. The equality is reached when $k=t=n/2 +o(n)$.
\end{proof}

\subsection {The second case}\label{s2}
$ $

The idea is to bound  $\P(M_n \in{\CM}_{k,t})$ differently when $2n-2t-k$ is big.  Let $M \in \CM_{k,t}$ and let $T_{t}$ be defined with respect to $M$. Recall that $T_{t}$ has $t$ columns, $2(n-t-1)$ rows, rank $k$ and the property that for any $1\leq i \le n-t-1$, if we delete its $i^{th}$ and $\left(n-t-1+i\right)^{th}$ rows, then the rank decreases by at least one. This motivates the following definition.

\begin{definition}
Let $M$ be a fixed $2m \times q$ matrix with $\pm1$ values. We say that $M$ has property $\CP$ if, for any $1\leq i \leq m$, by deleting both the $i^{th}$ row and the $\left(i+m\right)^{th}$ row, we reduce the rank of $M$ by at least one.
\end{definition}

\begin{definition}\label{alpha}
Let $A:=\big{\{}\beta \, \big |\, \P\left(M_n \text{ is C-normal}\right)\leq 2^{-(\beta+o(1))n^2} \big{\}}$. We define
	$$\alpha := \limsup_{\beta\in A} \beta-0.0001.  $$
\end{definition}

Lemma \ref{lemma3} implies that
$$ \alpha \geq 0.2499.$$

\begin{lemma}
\label{Lemma3.3}
Given $1\leq k,t \leq n$ we have that:
	$$ \P\left(M_n \in \CM_{k,t}\right)\leq 2^{(1-\alpha)t^2-k^2/2-n^2+nk+o(n^2)}.$$
\end{lemma}

\begin{proof}[Proof of Lemma \ref{Lemma3.3}]
The intuition is that, given a uniformly random $2(n-t-1)\times t$ matrix with $\pm 1$ entries and rank $k$, the probability that it has \textit{property} $\CP$ is very small for particular values of $k$ and $t$.  

Note that by Observation \ref{obs1} we have that the probability in question is zero unless $n-k -2 \leq t \leq n-k/2$. We start by making two observations.

\begin{observation} \label{aa}
$ $

\begin{enumerate}[label=(\alph*)]

\item \label{aa1}
Let $M$ be a $2m \times q$ matrix, then for any $1\leq i \leq m$, we can swap the $i^{th}$ row of $M$ with the $\left(m+i\right)^{th}$ row of M without changing its property $\CP$ status. 

\item \label{aa2}
Let $M$ be a $2m \times q$ matrix, then for any $1\leq i<j \leq m$, we can swap the $i^{th}$ row of $M$ with the $j^{th}$ row of $M$ and the $\left(m+i\right)^{th}$ row of $M$ with the $\left(m+j\right)^{th}$ row of $M$, without changing its property $\CP$ status.

\end{enumerate}

\end{observation}

Given a matrix $M$ of rank $k$, it would be more convenient to bound the probability of having property $\CP$ if the first $k$ rows were linearly independent. It turns out that we only lose a factor of $2^{o(n^2)}$ if we consider only such matrices. A precise statement is given in Claim \ref{claim10}.

\begin{definition} We say that a matrix $M$ of rank $k$ has property $\CF_k$ if it has property $\CP$ and its first $k$ rows are linearly independent.
\end{definition}

\begin{claim} \label{claim10} 	
Let $M_{m,q}$ be a $2m \times q$ random matrix with Rademacher entries which take the values $\pm 1$ with probability $1/2$. We have
$$\P(M_{m,q} \text{ has property } \CP \text{ and rank }k) \leq 2^{o(m^2)} \, \P(M_{m,q} \text{ has property } \CF_k). $$
\end{claim}

	\begin{proof}[Proof of Claim \ref{claim10}]
		Let $M$ be a fixed matrix of dimension $2m \times q$ and rank $k$ which has property $\CP$. We prove that we can apply a series of operations described in Observation \ref{aa} to reduce it to a matrix which has property $\CF_k$. Since we have at most $\left(2m\right)! = 2^{o\left(m^2\right)}$ ways to permute the rows of $M$, the conclusion follows. $ $\newline
		
		 Suppose that there exists a fixed matrix $M$, that cannot be reduced to one with property $\CF_k$ using only operations from Observation \ref{aa}. Let $i\leq k$ be the biggest index such that there exists a matrix $M'$, formed by applying such operations to $M$, and its $i^{th}$ row is the first row that is not linearly independent to the previous $i-1$ rows. $ $ \newline

		  If $i\le m$, then by property $\CP$, we know that if we delete both the $i^{th}$ row and the $\left(m+i\right)^{th}$ row from $M'$, then we decrease the rank of its row space by at least one. Since the $i^{th}$ row is in the span of the first $i-1$ rows, then we deduce that the $\left(m+i\right)^{th}$ row is linearly independent to the first $i-1$ rows. By Observation \ref{aa}\ref{aa1} we can swap the $i^{th}$ row with the $\left(m+i\right)^{th}$ row and still preserve property $\CP$. The first $i$ rows of the new matrix are linearly independent, which contradicts the maximality of $i$. $ $ \newline

		  If $k\ge i>m$, since  $\rank(M')=k$, then there exists $p$ with $2m\geq p>k$ such that the  $p^{th}$ row of $M'$ is not in the span of the first $k$ rows of $M'$. By Observation \ref{aa}\ref{aa2} we can swap the $i^{th}$ row with the $p^{th}$ row and the $\left(i-m \right)^{th}$ row with the $\left(p-m\right)^{th}$ row, creating a matrix whose first $i^{th}$ rows are linearly independent, which contradicts the maximality of $i$.

 		We conclude that
		$$\P(M_{m,q} \text{ has property } \CP \text{ and rank }k)\leq \P(M_{m,q} \text{ has property } \CF_k)2^{o(m^2)}.$$
		\end{proof}
	
\begin{lemma}\label{lemma3.4}
Let $M_{m,q}$ be a $2m \times q$ random matrix with Rademacher entries that take the values $\pm 1$ with probability $1/2$. We have
 	$$\P(M_{m,q}\text{ has property } \CF_k)\leq 2^{(2m-k)(k-m-q)+o(m^2)}.$$
\end{lemma}

\begin{proof}[Proof of Lemma \ref{lemma3.4}]	We start with some notation: for a matrix $M$, we write $M^{(i_1,...,i_j)}$ to denote the submatrix of $M$ created by removing its $i_1^{th},...,i_j^{th}$ rows. We also write $\rowsp\left(M\right)$ to denote the row space of a matrix $M$ and $\row_i(M)$ to denote the $i^{th}$ row of a matrix $M$.
$ $ \newline

Let us reveal the main idea behind the proof: we use a counting argument and we condition on the first $k$ rows of $M_{m,q}$. Define
	 $$K : = M_{m,q}^{(k+1,...,2m)}.$$ Note that if $M_{m,q}$ has property $\CF_k$, then it has rank $k$, hence all of its rows belong to the row space of $K$. Our argument is a slightly more sophisticated version of the above fact. $ $ \newline

\noindent \underline{\textbf{A condition argument}}. Let $M_{m,q}$ be a $2m \times q$ matrix with property $\CF_k$ and $j$ a positive integer such that $j\leq k$ and $k < j+m \leq 2m$. By property $\CF_k$, we have that the rank of $M_{m,q}^{(j,m+j)}$ is at most $k-1$. However,
 	$$ \rank \left( M_{m,q}^{(j,k+1,...2m)} \right) = \rank \left(K^{(j)} \right) = k-1, $$
\noindent which implies
	 $$\rowsp\left(M_{m,q}^{(j,m+j)}\right) = \rowsp\left(K^{(j)}\right)$$

\noindent or, equivalently,  

	\begin{equation}\label{new1}
	 \row_i(M_{m,q}) \in \rowsp\left( K^{(j)} \right) \text{ for any } k < i \leq 2m, \, i\not = m+j.
	 \end{equation}
\begin{observation} If $k<m$, by relation \eqref{new1} we have that $$\row_{2m}(M_{m,q}) \in \left( \bigcap_{j\leq k} \rowsp \left(K^{(j)}\right)\right) = \{ 0_{\R^q} \}$$
which is a contradiction. That is, $k\geq m$.
\end{observation}
	 
On the other hand, if $j\leq m$ and $j+m \leq k$, then the rank of $M_{m,q}^{(j,m+j)}$ could potentially be $k-2$ and we can not conclude that $ \row_i(M_{m,q}) \in \rowsp\left( K^{(j,j+m)} \right)$ for $ k < i \leq 2m$. However, note that by property $\CF_k$ 

	$$ \dim  \Span \left( \bigcup_{k < i \leq 2m} \row_i(M_{m,q})  \cup \rowsp \left( K^{(j,j+m)}\right) \right)  \leq k-1 = \dim \rowsp \left( K^{(j,j+m)} \right) + 1,$$
	
\noindent which means that if there exists $i_j$ such that $k<i_j\leq 2m$ and $\row_{i_j}(M_{m,q}) \not \in \rowsp \left(K^{(j,j+m)}\right)$, then  
	
		\begin{equation}\label{new2}
	 \row_i(M_{m,q}) \in \Span \left( \rowsp\left( K^{(j,m+j)} \right) \bigcup \row_{i_j}(M_{m,q})\right)\text{ for any } k < i \leq 2m, \, i\not = i_j.
	 \end{equation}
	
$ $ \newline
We further divide the matrices from $\CF_k$ into $(2m-k+1)^{m-k}$ categories $\CF_k^{(i_1,...,i_{k-m})}$, based on the smallest indices $(i_1, i_2, ..., i_{k-m})$  where $i_s > k$ for any $s\leq k-m$, where relation \eqref{new2} holds. We have
	
	\begin{equation}\label{new9}
	\row_i(M_{m,q}) \in \left\{
\begin{array}{ll}
      \Span \left( \rowsp\left( K^{(j,m+j)} \right) \bigcup \row_{i_j}(M_{m,q})\right)  & \text{ for } k < i \leq 2m, \, i > i_j. \\
      \Span \left( \rowsp\left( K \right) \right)  &  \text{ for } k < i \leq 2m, \, i = i_j \\
      \Span \left( \rowsp\left( K^{(j,m+j)} \right) \right) & \text{ for } k < i \leq 2m, \, i < i_j.
\end{array} 
\right. 
	\end{equation}
$ $ \newline

\noindent \underline{\textbf{Counting the matrices in  $\CF_k^{(i_1,...,i_{k-m})}$}}.  Fix $i > k$. We divide the set of indices $\{i_1,...,i_{k-m}\}$ into three subsets, $A_{i_>}$, $A_{i_=}$ and $A_{i_<}$, based on whether the indices are greater, equal or smaller then $i$. 
Equations \eqref{new1}  and \eqref{new2} imply that

	\begin{align}\label{new13}
	\row_i(M_{m,q})   \in   \bigcap_{\substack{k-m < j \leq  \min(k,m) \\ j \not = i-m}}  & \rowsp\left( K^{(j)} \right)  \bigcap_{ j  \in A_{i_<}}  \Span \left( \rowsp\left( K^{(j,m+j)} \right) \bigcup \row_{i_j}(M_{m,q})\right) \nonumber \\
				       & \qquad \bigcap_{ j  \in A_{i_>}}  \Span \left( \rowsp\left( K^{(j,m+j)} \right)\right) 
	\end{align}

	Let $x$ be a vector from the subspace of the right hand side of relation \eqref{new13} and let $x = \sum_{k} a_k \row_i(K)$ be the unique decomposition of $x$ in terms of the rows of $K$.

	Let $x \in  \bigcap_{ j  \leq k-m}   \Span \left( \rowsp\left( K^{(j,m+j)} \right) \bigcup \row_{i_j}(M_{m,q})\right) $ and let $x = \sum_{k} a_k \row_i(K)$ be the unique decomposition of $x$ in terms of the rows of $K$. 
	
	\begin{itemize}
		\item Let $j\in A_{i_>}$. Since $x \in  \Span \left( \rowsp\left( K^{(j,m+j)} \right) \right)$ we have $a_j = a_{j+m} = 0$ since the rows of $K$ are linearly independent. 
		\item Let $j\in A_{i_<}$. Since $x \in  \Span \left( \rowsp\left( K^{(j,m+j)} \right) \bigcup \row_{i_j}(M_{m,q})\right)$ we have:
			$$a_j \row_j(K) + a_{m+j}\row_{m+j}(K) = \row_{i_j}^{[j,m+j]}(M_{m,q}),$$
\noindent where by $\row_{i_j}^{[j,m+j]}(M_{m,q})$ we mean the projection of $a_j\row_j(K)+a_{m+j} \row_{m+j}(K)$ on $\row_{i_j}(M_{m,q})$.
	\end{itemize}
\noindent Thus, we have:
		\begin{equation}\label{new14}
			 \row_i(M_{m,q}) \in \Span \left(  \bigcup_{j\in A_{i_<}}  \row_{i_j} (M_{m,q}) \bigcup_{j \in A_{i_=}} \Big(\row_{j} (M_{m,q})\cup \row_{m+j} (M_{m,q})\Big) \bigcup \row_{i-m}(K) \right),
		\end{equation}
		which is a subspace of dimension at most $|A_{i_<} |+ |A_{i_=}|+1 = k-m +1 - |A_{i_>}|$. 
		
$ $\newline

\noindent We start by counting the matrices in  $\CF_k^{(i_1,...,i_{k-m})}$.

\begin{itemize}
	\item We can pick the elements from $K$ in $$2^{kq} \text{ ways. }$$ 
	\item By relation \eqref{new14} we can pick the elements from the $\row_i(M_{m,q})$, where $i > k$ in 
		$$2^{k-m+1 -|A_{i_>}|} \text{ ways. }$$
\end{itemize}

\noindent \underline{\textbf{Final counting}}. We conclude that:
	
	\begin{align*}
	 	\P(M_{m,q}\text{ has property } \CF_k) & \leq \sum_{(i_1,i_2,...,i_{k-m})} \P \left(M_{m,q} \text { has property } \CF_k^{(i_1,i_2,...,i_{k-m})} \right) \\
						& \leq (2m-k+1)^{m-k} \sup_{(i_1,i_2,...,i_{k-m})}  \P \left(M_{m,q} \text { has property } \CF_k^{(i_1,i_2,...,i_{k-m})} \right) \\
						& = 2^{o(m^2)}  \sup_{(i_1,i_2,...,i_{k-m})}  \frac{ 2^{\big( kq+(k-m+1)(2m-q)\big) -  \sum_{i}|A_{i_>}|}}{2^{2mq}} \\
						& \leq 2^{o(m^2)} \cdot 2^{ \big( kq+(k-m+1)(2m-q)\big) -  2mq} = 2^{(2m-q)(k-m-q)+o(m^2)}.
	\end{align*}
\end{proof}

\noindent \textit{Proof of Lemma \ref{Lemma3.3}.}
	\begin{align*}
	\P\left(M_n \in \CM_{k,t}\right) & \leq \sup_{T_t} \P\left(M_n\in \CM_{k,t}|T_t\right) \cdot \P\left(T_t \text{ has property } \CP \text{ and rank }k \right) \\
		&\leq \sup_{T_t} \P\left(M_n\in \CM_{k,t}|T_t\right) 2^{(2n-2t-k)(k-n)+o(n^2)},
	\end{align*}
	where the last step follows by Claim \ref{claim10} and Lemma \ref{lemma3.4}. Note that by the definition of $\alpha$,
	
	\begin{align*}
		\P\left(M_n\in \CM_{k,t}|T_t\right)  & \leq \sup_{T_t, M_n(<t;<t)} \Bigg( \P\big( M_n\in \CM_{k,t}|T_t, M_n(<t;<t) \big)\cdot \\
		&\qquad \cdot  \P\Big(M_n(<t;<t)\text{ is }M_n(<t;>t) M_n(<t;>t)^T- M_n(>t;<t)M_n(>t;<t)^T +\\
		& \hspace*{9cm} +C(<t;<t)\text{-normal}\Big) \Bigg)\\
		& \leq \sup_{T_t,M_n(<t;<t)} \left( \P\big( M_n\in \CM_{k,t}|T_t,M_n(<t;<t)\big)2^{-\alpha t^2 +o(n^2)}\right).
	\end{align*}

We keep the notation from Lemma \ref{lemma10}, that is
	$$D_i : = \{ M(i',i') \text{ where } 1\leq i' \leq n\} \cup \{ M(i',j') \text{ where either } i' \leq i  \text{ or } j' \leq i  \} .$$
By Lemma \ref{lemma10} we have
	\begin{align*}
	 \P\left( M_n\in \CM_{k,t}|T_t,M_n(\leq t,\leq t)\right) &\leq\sup_{D_t} \P\left( M_n\in \CM_{k,t}|D_t\right) \\
	&\leq 2^{-\left(\rank(T_t)+...+\rank(T_{2n-k-t})\right)-(k+t-n)^2}\\
	& \leq 2^{k^2/2+n^2+t^2-2nk+2kt-2nt+o(n^2)}.
	\end{align*}

\noindent Finally, we conclude
	\begin{align*}
	\P\left(M_n \in \CM_{k,t}\right) & \leq 2^{(2n-2t-k)(k-n)-\alpha t^2 +k^2/2+n^2+t^2-2nk+2kt-2nt+o(n^2)}\\
	& \leq 2^{(1-\alpha)t^2-k^2/2-n^2+nk+o(n^2)}.
	\end{align*}
\end{proof}

\subsection{Proof of Theorem 1.2 }\label{s3}

$ $

In this section, we put everything together to conclude the proof of Theorem 1.2 using a case analysis to improve the lower bound on $\alpha$. Recall that:
	$$ 2^{\left(-\alpha +o(1)\right)n^2} = \P(M_n \text{ is C-normal}) \geq \sup_{k,t} \P \left( M_n \in \CM_{k,t} \right),$$
thus, in order to improve the lower bound on $\alpha$ we improve the lower bound on $\sup_{k,t} \P \left( M_n \in \CM_{k,t} \right)$.
$ $\newline	
	
Define $f,g_1$ and $g_2$ by the following formulas: 
	\begin{align}
	&f(\alpha,n,k,t):=(1-\alpha)t^2-k^2/2-n^2+nk\label{111}\\
	&g_1(n,k,t):=t^2-3k^2+2kn+kt-2nt\label{222}\\
	&g_2(n,k,t):=n^2+k^2+t^2+kt-2kn-2nt\label{333}.
	\end{align}

By Lemma 3.1 and Lemma 3.3 we have
\begin{equation}\label{important2}
\P(M_n \in \CM_{k,t})\leq
\begin{cases}
 \min\big(2^{g_1(n,k,t)+o(n^2)},2^{f(\alpha,n,k,t)+o(n^2)}\big)  & \mbox{if }k\leq \frac {n}{2}\\
 \min\big(2^{g_2(,n,k,t)+o(n^2)} ,2^{f(\alpha,n,k,t)+o(n^2)}\big)  & \mbox{if }k\geq \frac {n}{2}.
\end{cases}
\end{equation}
	
	 For fixed $k$, both $g_1$ and $g_2$ are decreasing functions of $t$, while $f$ is increasing in $t$. It follows that the worst lower bound for $\alpha$ is achieved in one of the six boundary cases:
	$$\begin{cases}
	 	f(\alpha,n,k,t)=g_1(n,k,t), \,t=n-k/2 \text{ or } t=n-k & \mbox{when } k\leq n/2\\
	 	f(\alpha,n,k,t)=g_2(n,k,t), \,t=n-k/2 \text{ or } t=k & \mbox{when } k\geq n/2.
	\end{cases}$$

Since all the equations are homogeneous, we will assume that $n=1$ and we will analyze each of these extreme cases.

\begin{enumerate} [label=\bf{Case 1:}]
\item 
 $t=1-k$ and $k\leq 0.5$. Relations \eqref{111} and \eqref{222} become
			\begin{align*}
			f(\alpha,1,k,1-k) &= -\alpha -(1-2\alpha)k+(1-\alpha)k^2\\
			g_1(1,k,1-k) &= -1+3k-3k^2,
			\end{align*}
			\noindent which implies
			$$\alpha \geq \min_{k\in[0,0.5]} \left( \max \left(\frac{1-k}{2-k}, 1+3k^2-3k\right)\right)\geq 0.425.$$
\end{enumerate}

\begin{enumerate} [label=\bf{Case 2:}]
\item $t=1-k/2$ and $k\leq 0.5$. Relation  \eqref{222} implies
				$$-\alpha \leq g_1(1,k,1-k/2) =-1-13k^2/4+3k \leq -0.307.$$ 
\end{enumerate}

\begin{enumerate} [label=\bf{Case 3:}]
\item  $t=1-k/2$ and $k\geq0.5$. Relation \eqref{333} implies
				$$-\alpha \leq g_2(1,k,1-k/2)=\frac{3k^2}{4}-k \leq -0.3125 \text{ since }k\leq2/3.$$
\end{enumerate}

\begin{enumerate} [label=\bf{Case 4:}]
\item $t=k$ and $k\geq0.5$. Relations \eqref{111} and \eqref{333} become
			\begin{align*}
			f(\alpha,1,k,k) &= (1/2-\alpha)k^2+k-1\\
			g_2(1,k,k) &= 1+3k^2-4k,
			\end{align*}
			\noindent which implies
			$$\alpha \geq \min_{k\in[0.5,1]} \left(\max  \left( \frac{-1+k^2/2+k}{1-k^2}, 1+3k^2-4k\right)\right) \geq 0.323.$$
\end{enumerate}

\begin{enumerate} [label=\bf{Case 5:}]
\item $f (\alpha,1,k,t)= g_1(1,k,t)$ and $k \leq 0.5$. 
			Since $f (\alpha,1,k,t)= g_1(1,k,t)$ we get
				$$ k = \frac{t+1+\sqrt {(t+1)^2-5(4t-2-2\alpha t^2)}}{5}.$$
			Hence,
			$$ \alpha \geq \min_{t} f\left(\alpha,1,\frac{t+1+\sqrt {(t+1)^2-5(4t-2-2\alpha t^2)}}{5},t\right),$$
			which leads to 	
				$$\alpha \geq 0.302.$$
\end{enumerate}

\begin{enumerate} [label=\bf{Case 6:}]
	\item $f(\alpha,1,k,t)=g_2(1,k,t)$ and $k \geq 0.5$. Since $f(\alpha,1,k,t)=g_2(1,k,t)$ we get
		$$ k = \frac{3-t - \sqrt {(t-3)^2 - 3 (4-4t+2\alpha t^2)} }{3}.$$
		Hence,
		$$\alpha \geq \min_t f\left(\alpha,1,\frac{3-t - \sqrt {(t-3)^2 - 3 (4-4t+2\alpha t^2)} }{3},t\right),$$
		which leads to 
		$$ \alpha \geq 0.307.$$
\end{enumerate}

		It follows immediately from the above cases that
		$$\nu_n  \leq   2^{-\left(0.302+o(1)\right)n^2}.$$

\noindent \textbf{Acknowledgment.} We would like to thank the referees for their helpful comments.

$ $


\begin{thebibliography}{1}
   \bibitem{Komlos} J. Koml\"os, On the determinant of $(0,1)$ matrices, {\em Studia Sci. Math. Hungar.} 2 (1967) 7-21.

  \bibitem{BVW}  J. Bourgain, V.  Vu, P. M. Wood, On the singularity probability of discrete random matrices, {\em J. Funct. Anal.} 258 (2010), no. 2, 559-603.

   \bibitem{Tao} T. Tao, Topics in Random Matrix Theory, {\em American Mathematical Society}, 2012.
   \bibitem{RB} R. Bhatia, Matrix Analysis, {\em Springer}, 1996.
   \bibitem{TV1} T. Tao and V. Vu, The Littlewood-Offord Problem in High Dimensions and a Conjecture of Frankl and Furedi, {\em Combinatorica}, 2002.
    \bibitem{KKS} J. Kahn, J. Komlos and E. Szemeredi, On the Probability that a Random $\pm 1$-matrix is Singular, {\em American Mathematical Society}, 1995.
     \bibitem{TV2} T. Tao and V. Vu, On the Singularity Probability of Random Bernoulli Matrices, {\em American Mathematical Society}, 2007.
    \bibitem{TV3} T. Tao and V. Vu, Random Matrices have Simple Spectrum, {\em Combinatorica } 37  no. 3, 539-553 (2017).
    
    \bibitem{VN} H. Nguyen and V. Vu,  Small Ball Probability, Inverse Theorems, and Applications, {\em Erd\"os Centennial}, 2013.
    \bibitem{NVT} H. Nguyen, V. Vu, T. Tao, Random Matrices: Tail Bounds for Gaps between Eigenvalues,  {\em Probab. Theory Related Fields}  167 (2017), no.3-4, 777-816.
    \bibitem{Odl} A. Odlyzko, On subspaces spanned by random selections of $\pm1$ vectors.   {\em Journal of Combinatorial Theory} Series A, 47(1):124-133, 1988.
  \end{thebibliography}
\end{document}